\documentclass[a4paper,11pt,times,authoryears]{article}
\usepackage{authblk}
\usepackage[authoryear]{natbib}
\usepackage{bm}
\usepackage{amsmath}
\usepackage{amsthm}
\usepackage{mathtools}
\usepackage{amssymb,amsfonts}
\usepackage{graphicx}
\usepackage{mathrsfs}
\usepackage{multirow, multicol}
\usepackage{enumitem, array}
\usepackage{subcaption}
\usepackage{epstopdf}
\usepackage{float}
\usepackage{enumerate}
\usepackage{booktabs}
\usepackage{setspace}
\usepackage{adjustbox}
\usepackage{pdflscape}
\usepackage{verbatim}

\usepackage[colorlinks,citecolor=blue]{hyperref}
\usepackage{rotating}

\usepackage{fullpage}
\usepackage{setspace}
\usepackage{lscape}
\usepackage{booktabs}
\usepackage{setspace}
\usepackage{tocbibind}
\usepackage{color}
\usepackage{commath} 
\usepackage{algorithm}
\usepackage{rotating}
\usepackage[noend]{algpseudocode}

\newtheorem{Theorem}{Theorem}[section]
\newtheorem{definition}[Theorem]{Definition}

\newtheorem{Lemma}[Theorem]{Lemma}

\usepackage{geometry}
\def\bD{{\mathbf D}}

\def\bV{{\mathbf V}}

\def\bI{{\mathbf I}} 
\def\bK{{\mathbf K}}
\def\bM{{\mathbf M}}

\def\bR{{\mathbf R}}
\def\bS{{\mathbf S}}

\def\bU{{\mathbf U}}
\def\bQ{{\mathbf Q}}
\def\bW{{\mathbf W}}
\def\hW{\widehat{\bW}}

\def\bX{{\mathbf X}}

\def\ba{{\mathbf a}}
\def\bb{{\mathbf b}}

\def\bx{{\mathbf x}}

\def\bz{{\mathbf z}}

\def\bY{{\mathbf Y}}

\def\balpha{{\boldsymbol{\alpha}}}

\def\ahat{\widehat{\balpha}}
\def\ah{\widehat{\alpha}}
\def\amle{\ahat_{\rm MLE}}

\def\bLam{\boldsymbol{\Lambda}} 

\def\bLk{\bLam_{k^+}}
\def\bLk{\bLam_{k^+}^{-1}}
\def\bLd{\boldsymbol{\Lambda}_{d^-}}

\def\bbeta{{\boldsymbol{\beta}}}
\def\bmu{{\boldsymbol{\mu}}}

\def\hbbeta{\widehat{\boldsymbol \beta}}
\def\hbz{\widehat{\bz}}

\def\no{\nonumber}

\def\bzero{{\boldsymbol 0}}

\def\bhat{\widehat{\boldsymbol \beta}}

\def\ll{\left(}
\def\rr{\right)}

\def\bVd{\bV_{d^-}}

\def\bVk{\bV_{k^+}}

\def\mle{\bhat_{\rm MLE}}
\def\lte{\bhat_{\rm LTE}}
\def\LTE{\bhat_{\rm LTE}}

\def\AULE{\bhat_{\rm AULE}}
\def\AULTE{\bhat_{\rm AULTE}}
\def\MAULTE{\bhat_{\rm MAULTE}}

\newtheorem{theorem}{Theorem}

\begin{document}
\begin{small}
\title{Almost Unbiased Liu Type Estimator in Bell Regression Model: Theory, Simulation and Application}
\author{{Caner Tan\i{}\c{s}$^\dag$ and Yasin Asar$^{\ddag}$ } \vspace{.5cm} \\$^{\dag}$Department of Statistics, \c{C}ank{\i}r{\i} Karatekin University \\
e-mail: canertanis@karatekin.edu.tr, caner.tanis@gmail.com
\\
$^{\ddag}$Department of Mathematics and Computer Sciences,\\ Necmettin Erbakan University, Konya, Turkey\\
e-mail: yasar@erbakan.edu.tr, yasinasar@hotmail.com}

\date{}
\maketitle
\begin{abstract}
In this paper, we gain the new almost unbiased Liu-type estimators to literature for the Bell regression model. We provide the superiority of the proposed estimator to its competitors such as the maximum likelihood estimator and Liu-type estimators via some theorems. We also design an extensive Monte Carlo simulation study to show that the proposed estimators outperforms the competitors in terms of mean squared error theoretically. Finally, we present a real data study to assess the performance of the introduced estimators in modeling real-life data. The findings of both the simulation and the empirical study demonstrate that the proposed regression estimators surpasses its competitors based on the mean square error criterion.

\end{abstract}
\end{small}


\onehalfspacing
\noindent
\section{Introduction}
\label{sec1}

In recent decades, count regression models have attracted significant attention because of their broad applications in various fields such as epidemiology, economics, reliability, and social sciences. A well-known problem in modeling count data is the presence of over-dispersion or under-dispersion, which often leads to inadequacies in the classical Poisson regression model. To solve this problem, several flexible alternatives have been suggested, including NB regression, Poisson-Inverse Gaussian (PIG) regression, Waring regression \citep{waring}, Berg regression \citep{berg}, and Conway-Maxwell-Poisson (COMP) regression \citep{conway}. Each of these models provides diverse mechanisms to account for dispersion, either by introducing additional parameters or by adopting more flexible distributional assumptions. Among these, the recently developed Bell regression model offers an attractive alternative due to its ability to simultaneously handle both under and over-dispersion in a unified structure, while retaining computational tractability and interpretability \citep{castellares2018}. Inspired by these characteristics, in this paper, we focus exclusively on the Bell regression model. While other models could also provide reasonable fits, our intention is to highlight the properties and inferential advantages of Bell regression.

On the other hand, in multiple regression models, the variation in the dependent variable is explained using independent variables. When independent variables are highly correlated with each other, interpreting the estimates of individual parameters becomes challenging \citep{hoerl1970}. This high correlation between the explanatory variables was first identified by \citet{frisch1934} and termed as multicollinearity. Furthermore, the issue of multicollinearity results in unreliable hypothesis tests and larger confidence intervals for the estimated parameters \citep{amin2022b}. 

To address these challenges associated with multicollinearity, several techniques have been proposed in the literature. One of the popular methods is the use of the ridge estimator (RE), which was first introduced by \citet{hoerl1970}. Several studies on ridge estimators can be found in the literature. For instance, \citet{kibria2003} introduced new estimators for generalized ridge regression analysis and evaluated their performance using simulation studies. The results indicated that, under certain conditions, the proposed estimators outperformed ordinary least squares and other existing popular estimators based on the mean squared error criterion.

The maximum likelihood estimator (MLE) is used to estimate the unknown parameters of the regression coefficients in Bell regression. However, MLE is considered inappropriate due to its significant disadvantages in the presence of multicollinearity \citep{amin2020a}. The disadvantages of the MLE in the cases of multicollinearity include the increase in variance and the standard error of the estimated regression coefficients, leading to inconsistent estimates \citep{ayinde2018, maj2022}. \citet{amin2020a} suggested a ridge estimator in Bell regression, demonstrating through simulation studies that the proposed ridge estimator is much more robust than the MLE under specific conditions. Other references related to RE include, \citet{alkhamisi2006},  \citet{kibria2015}, \citet{yuz2018}, \citet{asar2022}, \citet{tharshan2024}.

Another method used to deal with the problem of multicollinearity in multiple regression models is the Liu estimator (LE), introduced by \citet{liu1993}. The LE has the advantage of its parameter $d$, here unlike the ridge regression, the estimated parameters are a linear function of $d$. In the methodology of biased estimators, the parameters $d$ and $k$ play crucial roles as biasing or tuning parameters within the estimators, specifically designed to address the challenges posed by multicollinearity. The parameter $d$, characteristic of Liu estimator, controls the shrinkage of the parameter estimates towards the least squares estimator \citep{liu1993}, thereby influencing the trade-off between bias and variance. Similarly, the parameter $k$, often associated with ridge estimator, serves an analogous function by adding a penalty to the least squares objective function, which also leads to biased but more stable estimates. By carefully selecting these parameters, these estimators aim to achieve a reduced mean squared error (MSE), particularly when independent variables are highly correlated. An increase in the values of $d$ or $k$ generally introduces a larger bias into the estimates, but simultaneously leads to a more substantial reduction in variance, which can be highly beneficial in scenarios with severe multicollinearity. Thus, understanding the influence of these parameters is vital for optimizing the performance and reliability of the proposed estimators. \citet{maj2022} proposed a new LE to address the issue of multicollinearity in Bell regression.

Furthermore, the Liu-type estimator (LTE), which is a two-parameter estimator, was proposed by \citet{liu2003} as a solution to the problem. \citet{bulut2021} and \citet{isiler2022} introduced an LTE in Bell regression and concluded that the newly introduced estimator outperformed the ridge and the LE in terms of MSE under certain conditions. Other references related to LE and LTE include \citet{man2011liu}, \citet{man2012liu},  \citet{asar2018liu}, \citet{amin2025} and \citet{tanis2024}.

On the other hand, \citet{kadiyala1984} proposed an almost unbiased estimator to address the problem of multicollinearity. This approach offers not only an unbiased estimator, but also an efficient one compared to the ordinary least squares estimator. Several studies in multiple regression models have employed almost unbiased estimators to address the multicollinearity problem. Some of these studies are as follows: \citet{amin2020b} introduced an almost unbiased ridge regression estimator (AURE) in gamma regression; \citet{altaweel2020} proposed several AUREs in the zero-inflated negative binomial regression models; \citet{noeel2021} introduced AUREs in the count regression models such as Poisson and negative binomial (NB); \citet{asar2022auliutype} suggested an almost unbiased Liu-type estimator (AULTE) in the gamma regression models; and \citet{tanis2024aure} proposed an AURE for the Bell regression model.


This study aims to introduce a new almost unbiased Liu-type estimators as alternatives to the existing LTE within the Bell regression model. The proposed estimator is compared to competitors, such as the MLE and LTE, using both scalar and matrix MSE criteria. Moreover, this paper seeks to validate the theoretical results through simulation studies and real data analysis to demonstrate the proposed estimators' advantages over the alternatives.

The paper is structured as follows: Section \ref{sec:bell} outlines the fundamentals of the Bell regression model, defines the LTE, and introduces the new estimators AULTE and modified AULTE (MAULTE). Section \ref{sec:comp} presents theoretical comparisons among the examined estimators. Section \ref{sec:sim} includes a comprehensive Monte Carlo simulation study to evaluate the performance of the introduced estimators and their competitors (MLE and LTE) based on the simulated mean squared error and squared bias (SB) criteria. In Section \ref{sec:data}, the superiority of the suggested estimators is illustrated using a real-world data example. The paper concludes with final remarks in Section \ref{sec:conc}.

\section{Bell Regression Model}\label{sec:bell}
Bell distribution is introduced by \cite{bell1934a,bell1934b}. The probability mass function (pmf) of the Bell distribution is given by
\begin{equation}
P\left( Y=y|\tau \right) =\frac{\tau ^{y}\exp \left\{ -\exp \left(
\tau \right) +1\right\} B_{y}}{y!},y=0,1,2,...
\end{equation}
where $\tau >0,$ $B_{y}$ denotes Bell numbers, defined as follows:
\[
B_{y}={e}^{-1}\sum\limits_{k=0}^{\infty }\frac{k^{y}}{k!}. 
\]
The corresponding mean and variance are 
\begin{equation}
E\left( Y\right) =\tau \exp \left( \tau \right) ,
\end{equation}
and
\begin{equation}
Var\left( Y\right) =\tau \left( 1+\tau \right) \exp \left( \tau
\right)
\end{equation}
respectively. 

Bell regression model is introduced by \cite{castellares2018}. The Bell regression model is summarized as follows: 
Let $\mu =\tau \exp \left( \tau \right)$ and $\tau
=W_{0}\left( \mu \right),$ where $W_{0}\left( \mu \right) $ is the Lambert
function. In this case, the pmf of the Bell distribution can be written as
\begin{equation}
\label{bell:pmf}
P\left( Y=y|\mu \right) =\frac{W_{0}\left( \mu \right) ^{y}\exp \left\{
1-\exp \left( W_{0}\left( \mu \right) \right) \right\} B_{y}}{y!},y=0,1,2,...
\end{equation}
The corresponding mean and variance are given by
\begin{equation}
E\left( Y\right) =\mu,
\end{equation}
and
\begin{equation}
Var\left( Y\right) =\mu \left( 1+W_{0}\left( \mu \right) \right) .
\end{equation}
where $\mu >0$ and $W_{0}\left( \mu \right) >0.$ Therefore, it is clear that $%
Var\left( Y\right)>E\left( Y\right).$ Thus, we conclude that the Bell distribution can be potentially suitable for modelling over-dispersed count data. 

The Bell distribution is particularly useful for analyzing count data that exhibit over-dispersion, as its variance is always greater than its mean. This feature makes it well-suited for situations where the standard Poisson model is inadequate to capture the observed variability. As a member of the exponential family of one parameter, the Bell distribution maintains a simple formulation while offering greater flexibility. It also approximates the Poisson distribution for small parameter values and possesses the desirable property of infinite divisibility. Collectively, these characteristics highlight the Bell distribution as a compelling alternative to modeling over-dispersed count data \citep{castellares2018}.

Now, let $Y_{1},Y_{2},...,Y_{n}$ be $n$ independent random variables from the Bell distribution given in (Eq. \ref{bell:pmf}) and the mean of $y_{i}$ fulfills the following functional relation:
$$g\left( \mu_{i}\right) =\eta_{i}=\bx_{i}^\top\bbeta, i=1,2,...,n$$
where $\mu_{i}$ is the mean of $Y_{i}\sim Bell(W_{0}(\mu _{i}))$, for $i=1,2,...,n$.

$\bbeta =\left( \beta _{1},\beta _{2},...,\beta _{p}\right) ^{\top}\in 
\mathbb{R}^{p}$ is a $p$-dimensional vector of regression coefficients $\left(p<n\right) $, $\eta_{i}$ is the linear estimator and $\bx_{i}^{\top}=\left( x_{i1},x_{i2},...,x_{ip}\right) $ refers to observations on $p$ known covariates. 

The variance of $Y_{i}$ depends on $\mu_{i}$ and, by extension, on the covariate values. As a result, it is common for models to incorporate non-constant response variances. We assume that the link function $g:\left( 0,\infty \right) \rightarrow\mathbb{R}$ is strictly monotonic and two times differentiable. Several alternatives for the mean link function exist. For example, useful link functions include logarithmic $g\left( \mu \right) =log\left( \mu \right) $, square root $g\left( \mu \right) =\sqrt{\mu }$ and identity $g\left( \mu \right) =\mu $ (with particular attention to ensuring the positivity of the estimates), among others; see also \cite{mcnelder}. 

The maximum likelihood method is used estimate the parameter vector $\bbeta$. The log-likelihood function, excluding constant terms, is expressed as follows
\[
\ell \left( \beta \right) =\sum\limits_{i=1}^{n}\left[ Y_{i}\log \left(
W_{0}\left( \mu _{i}\right) \right) -\exp \left( W_{0}\left( \mu _{i}\right)
\right) \right],
\]%
where $\mu_{i}=g^{-1}\left( \eta _{i}\right) $ is a function of $\bbeta $, and $g^{-1}\left( .\right) $ is the inverse of $g\left( .\right) $. The score
function is given by the $p$-vector
\[\mathbf{U}\left( \bbeta \right) \mathbf{=X}^\top\mathbf{W}^{1/2}\mathbf{V}^{-1/2}\left( \bY-\bmu \right)\]%
where the model matrix $\mathbf{X=(x}_{1}\mathbf{,x}_{2}\mathbf{,...,x}_{n}%
\mathbf{)}^\top$ has full column rank,
$\bW=diag\left\{w_{1},w_{2},...,w_{n}\right\} $, $\bV=diag\left\{
V_{1},V_{2},...,V_{n}\right\} $, $\bY=\left(
y_{1},y_{2},...,y_{n}\right)^\top$, $%
\bmu=\left(\mu_{1},\mu_{2},...,\mu_{n}\right) ^\top$, and
\[
w_{i}=\frac{\left( d\mu_{i}/d\eta _{i}\right) ^{2}}{V_{i}},V_{i}=\mu _{i}\left[ 1+W_{0}(\mu _{i})\right] ,i=1,2,...,n,
\]
where $V_{i}$ is the variance function of $Y_{i}$. 

The Fisher information matrix for $\bbeta $ is expressed as $\bK\left( \bbeta\right) =\mathbf{X}^\top\mathbf{WX}$. The MLE $\widehat{\bbeta}=\left( \hat{\beta}_{1},\hat{\beta}_{2},...,\hat{\beta}_{p}\right)^\top$ for $\bbeta =\left( \beta _{1},\beta _{2},...,\beta _{p}\right)^\top$ is derived as the solution to $\bU\left( \widehat{\bbeta}\right) =\bzero_{p}$, where $\bzero_{p}$ represents a p-dimensional zero vector. Unfortunately, there is no closed-form solution for the MLE $\widehat{\bbeta}$, necessitating numerical computation. Techniques such as the Newton-Raphson iterative method can be employed for this purpose. Alternatively, the Fisher scoring method can be utilized to estimate $\bbeta$ through iterative resolution of the corresponding equation.
\begin{equation}
\label{eq7}
\bbeta ^{\left( m+1\right) }=\left(\bX^\top \bW^{\left( m\right) }X\right) ^{-1}%
\mathbf{X}^\top\mathbf{W}^{\left( m\right) }\mathbf{z}^{\left( m\right) },
\end{equation}
where $m=0,1,...$ is the iteration counter, $\bz=\left(z_{1},z_{2},...,z_{n}\right) ^\top=\boldsymbol{\eta} +\bW^{-1/2}\mathbf{V}^{-1/2}\left( \bY-\bmu \right) $ actions as a modified response variable in Eq. (\ref{eq7}) whereas $\mathbf{W}$ is a weight matrix, and $\boldsymbol{\eta} =\left( \eta _{1},\eta_{2},...,\eta _{n}\right)^\top.$ The maximum likelihood estimate $\mle$ can be iteratively derived by using Eq. (\ref{eq7}) through any software program with a weighted linear regression routine, such as R software \citep{castellares2018}. In this regard, the MLE of $\bbeta$ in the Bell regression  obtained using the IRLS algorithm
at the final step is given as follows:
\begin{eqnarray}\label{mle}
    \mle = \left(\bX^\top\hW \bX\right)^{-1} \bX^\top \hW \hbz
\end{eqnarray}
where $\hW$ and $\hbz$ are computed at final iteration.

Let $\bQ^{\top}\bX^{\top}\widehat{\bW}\bX\bQ=\bLam=diag\ll\lambda_{1},\lambda_{2},\ldots,\lambda_{p}\rr$, where $\lambda_{1}\geq\lambda_{2}\geq\ldots\geq\lambda_{p}> 0$ are the ordered eigenvalues of $\bX^{\top}\hW\bX$ and the $p\times p$ matrix $\bQ$ whose columns are the normalized eigenvectors of $\bX^{\top}\hW\bX$. Here, we define $\balpha=\bQ^{\top}\bbeta$, and the canonical form of MLE can be expressed by $\amle=\bQ^{\top}\mle$. 

The scalar MSE of $\mle$ can be given as \citep{maj2022}
\begin{align}
{\rm MSE}\ll \mle \rr&=E\ll(\mle-\bbeta)^{\top}(\mle-\bbeta)\rr \nonumber\\
&=~tr\ll(\bX^{\top}\hW\bX)^{-1}\rr=tr\ll\bQ\bLam^{-1}\bQ^\top\rr=\sum_{j=1}^{p}
\frac{1}{\lambda_{j}}
\label{mse:mle}
\end{align}
where, $tr(.)$ denotes the trace operator, and $\lambda_{j}$ represents the jth eigenvalue of the weighted cross-product matrix $\bX^{\top}\hW\bX$. It is evident from Equation \eqref{mse:mle} that the variance of the MLE may be adversely influenced by the ill-conditioning of the data matrix $\bX^{\top}\hW\bX$, a phenomenon commonly referred to as multicollinearity. For an in-depth discussion of this collinearity issue within generalized linear models, refer to the works of \citet{seger1992} and \citet{mac-put}.

\subsection{The Bell LTE}\label{sec:lte}
The LTE is introduced by \citet{bulut2021} and \citet{isiler2022} for the Bell regression model as follows:
\begin{eqnarray}
\label{LTE}
\lte =  \bVk^{-1}\bVd\mle\\ 
\end{eqnarray}
where $-\infty <d<\infty$, $k>0$, $\bV =\bX^{\top}\hW\bX$, $\bVd =\left(\bV -d\bI \right)$ and $\bVk = \left( \bV +k\bI \right)$. The covariance matrix and bias vector of LE are given by \citet{bulut2021} and \citet{isiler2022} respectively as follows:
\begin{eqnarray}
{\rm Cov}\left(\LTE\right) &=& \bVk^{-1} \bVd \bV^{-1}\bVd\bVk^{-1},
\label{cov:LTE}\\
{\rm bias}\left(\lte\right) &=& -\left(d+k\right)\bVk^{-1}\bbeta. \label{bias:LTE}
\end{eqnarray}
In order to derive the matrix MSE (MMSE) and MSE functions of the estimators, we make use of the followings:
\begin{eqnarray}
\bVk &=& \bQ\ll\bLam + k\bI\rr\bQ^{\top},
\label{vk}\\
\bVd &=& \bQ\ll\bLam - d\bI\rr\bQ^{\top}, 
\label{vd}\\
\bVk^{-1} \bVd \bV^{-1}\bVd\bVk^{-1} &=& \bQ\ll(\bLam + k\bI)^{-1} (\bLam-d\bI) \bLam^{-1}(\bLam-d\bI)(\bLam + k\bI)^{-1}\rr\bQ^{\top}. 
\label{vkd}
\end{eqnarray}
where $\bLk=\ll\bLam + k\bI\rr^{-1}$ and $\bLd=\bLam-d\bI$.

Therefore, the MMSE and MSE functions of LTE are computed by \citet{bulut2021} and \citet{isiler2022} as
\begin{eqnarray}
{\rm MMSE}\left(\LTE\right)
&=& \bVk^{-1} \bVd \bV^{-1}\bVd\bVk^{-1} + \left(d+k \right)^2 \bVk^{-1}\bbeta\bbeta^{\top} \bVk^{-1},  \no\\
&=& \bQ \ll \bLk \bLd \bLam^{-1} \bLd \bLk + (k+d)^2 \bLk \balpha\balpha^\top \bLk \rr \bQ^\top \label{mmse:LTE}\\
{\rm MSE}\left(\LTE \right)  &=& tr\ll{\rm MMSE}\left(\LTE\right)\rr\no\\
&=& \bQ \ll tr\ll \bLk \bLd \bLam^{-1} \bLd \bLk + (k+d)^2 \bLk \balpha\balpha^\top \bLk \rr \bQ^\top\rr\no\\
&=&\sum_{j=1}^{p}\frac{\left( \lambda_{j}-d\right)^2 }{\lambda_{j}\left( \lambda_j+k \right)^2}+\left( d+k\right)^2\sum_{j=1}^{p}\frac{\alpha_j^2}{\left( \lambda_j+k \right)^2}.\label{mse:LTE}
\end{eqnarray}


\subsection{The new Bell AULTE}\label{sec:new}

In this subsection, we introduce a new AULTE as an alternative to LTE and MLE in Bell regression model. In a practical sense, an almost unbiased estimator represents a strategic refinement in the bias-variance trade-off, particularly crucial in the presence of multicollinearity. While an ideally unbiased estimator's expected value perfectly matches the true parameter, it can suffer from inflated variance in ill-conditioned data, leading to imprecise and unreliable estimates. Almost unbiased estimators, such as the ones proposed, intentionally introduce a small, controlled amount of bias. The key practical implication is that this carefully managed bias allows for a substantial reduction in the estimator's variance. Consequently, despite not being strictly unbiased, these estimators achieve a significantly lower MSE, providing more stable, precise, and practically reliable estimates, especially when exact unbiasedness is secondary to minimizing overall estimation error in real-world applications with finite samples.

Therefore, we provide a definition of being an almost unbiased estimator and make use of it to define our proposed estimators.

\begin{definition}
\label{def:almost}
\cite{almost} Suppose
$\bbeta$ is a biased estimator of parameter vector $\beta$, and
if the bias vector of $\hat{\bbeta}$ is given by
$b(\hat{\bbeta})=E(\hat{\bbeta})-\bbeta=\bR\bbeta$, which shows that
$E(\hat{\bbeta}-\bR\bbeta)=\bbeta$,
then we call the estimator $\tilde{\bbeta}=\hat{\bbeta}-\bR\hat{\bbeta}=(\bI-\bR)\hat{\bbeta}$ is the almost unbiased estimator based on the biased estimator $\hat{\bbeta}$.
\end{definition}

In the line of the Definition \ref{def:almost}, the AULTE can be defined as follows: 
\begin{align}
    \AULTE &= \LTE - \ll- (d+k)\bVk^{-1}\LTE \rr\nonumber\\
    &= \bVk^{-1}\bVd\mle + (d+k)\bVk^{-1}\bVk^{-1}\bVd\mle  \nonumber\\
    &=\left( \bI+\left( k+d\right)\bVk^{-1}\right)\bVk^{-1}\bVd \mle\nonumber\\
   &=\left( \bI+\left( k+d\right)\bVk^{-1}\right)\left( \bI-\left( k+d\right)\bVk^{-1}\right)\mle\nonumber\\
   &= \left( \bI-\left( k+d\right) ^{2}\bVk^{-2}\right) \mle
\end{align}
where $-\infty<d<\infty$ and $k>0$.
Based on our literature review, the AULTE has neither been proposed nor examined within the context of the Bell regression model.

The covariance matrix and bias vector of the AULTE are
\begin{align}
Cov\left( \AULTE\right)&= \bK \bV^{-1}\bK ^{\top}\no\\
&= \bQ \ll \left( \bI-\left( k+d\right) ^{2}\bVk^{-2}\right) \bLam^{-1} \left( \bI-\left( k+d\right) ^{2}\bVk^{-2}\right) \rr \bQ^\top
\label{cov:AULTE}
\end{align}
and
\begin{align}
\label{bias: aulte}
Bias\left( \AULTE\right) =-\left( d+k\right)^{2}\bVk^{-2}\bbeta ,
\end{align}
respectively, where $\bK=\bI-\left( k+d\right) ^{2}\bVk^{-2}$. Thus, the MMSE and MSE of AULTE are respectively computed as
\begin{align}
MMSE\left( \AULTE\right) &= Cov\left( \AULTE\right)+Bias\left( \AULTE\right) Bias\left(\AULE\right)^{\top} \nonumber\\
&=\bK \bV^{-1}\bK ^{\top}-\left( k+d\right)^{2}\left( \bVk\right) ^{-2}\bbeta \left( -\left( k+d\right)^{2}\left( \bVk\right) ^{-2}\bbeta\right)^{\top}\nonumber\\
&=\bK \bV^{-1}\bK ^{\top}+\left( k+d\right)^{4}\bVk^{-2}\bbeta\bbeta^{\top}\bVk^{-2}\no\\
&= \bQ \ll \left( \bI-\left( k+d\right) ^{2}\bVk^{-2}\right) \bLam^{-1} \left( \bI-\left( k+d\right) ^{2}\bVk^{-2}\right) \rr \bQ^\top \no\\
&~+\bQ \ll (k+d)^4 \bVk^{-2} \balpha \balpha^\top \bVk^{-2}\rr \bQ^\top
\end{align} and
\begin{eqnarray}\label{eqMSE}
MSE\left( \AULTE\right)  &=& tr\left( Cov\left( \AULTE\right)\right)
+Bias\left( \bhat\right)^{\top} Bias\left( \bhat\right)\nonumber \\
&=& \sum_{j=1}^{p} \frac{1}{\lambda_j } \ll 1-\frac{(k+d)^2}{(\lambda_j + k)^2}\rr \sum_{j=1}^{p} \left( \frac{(k + d) ^{2}\alpha_j}{(\lambda_j + k)^2} \right)^2 \nonumber\\
&=& \sum_{j=1}^{p} \frac{(\lambda_j - d)^2}{\lambda_j (\lambda_j + k)^2} + + (k+d)^4\sum_{j=1}^{p} \left( \frac{\alpha_j^2}{(\lambda_j + k)^4} \right) \nonumber\\
&=&\sum\limits_{j=1}^{p}\frac{\left(
\lambda _{j}-d\right) ^{2}\left( \lambda _{j}+d+2k\right) ^{2}}{\lambda
_{j}\left( \lambda _{j}+k\right) ^{4}}+\left( k+d\right)
^{4}\sum\limits_{j=1}^{p}\frac{\alpha _{j}^{2}}{\left( \lambda _{j}+k\right)
^{4}}.
\label{mse:AULTE}
\end{eqnarray}


\subsection{The new Bell MAULTE}\label{sec:maulte}

When introducing a modified version of an estimator, such as the modified almost unbiased Liu-type estimator (MAULTE), the primary motivation is often to enhance its statistical properties beyond what the original estimator offers. Even an almost unbiased estimator, while significantly reducing variance, might retain a residual bias that could be further minimized. Therefore, the modification is typically aimed at achieving an even closer approximation to unbiasedness while maintaining or further improving the MSE, particularly in challenging conditions like severe multicollinearity. This could involve fine-tuning the structure of the estimator to make it more robust, providing more stable performance across a wider range of data characteristics, or offering a superior bias-variance trade-off compared to its predecessor. By introducing such a modification, the aim is to develop an estimator that provides more accurate and reliable parameter estimates in practical applications by further mitigating the adverse effects of collinearity.

In this subsection, following \citet{batah, khurana, yildiz} we suggest MAULTE by modifying AULTE for Bell regression model as follows
\begin{align}
    \MAULTE&=\left( \bI-\left( k+d\right) ^{2}\bVk^{-2}\right)\left(\bI-\left( k+d\right) \bVk^{-1}\right) \mle
\end{align}
where $-\infty<d<\infty$ and $k>0$.

The covariance matrix and bias vector of the MAULTE are
\begin{align}
Cov\left( \MAULTE\right)&= \bM \bV^{-1}\bM ^{\top},
\end{align}
and
\begin{align}
\label{bias:maulte}
Bias\left( \MAULTE\right) &=E\left( \MAULTE\right)
-\bbeta \nonumber\\
&=-\left( d+k\right)\left(\bI+\left( d+k\right)\bVk^{-1}-\left( d+k\right)^{2} \bVk^{-2}\right)\bVk^{-1} \bbeta ,
\end{align}
respectively, where 
\begin{align}
  \bM&=\left(\bI-\left( k+d\right) ^{2}\bVk^{-2}\right)\left(\bI-\left( k+d\right)\bVk^{-1}\right)\nonumber\\
  &= \bI - (k+d)\bVk^{-1}-(k+d)^2\bVk^{-2}+(k+d)^3\bVk^{-3}.\nonumber
\end{align}
In this regard, the MMSE and MSE of MAULTE are respectively given by
\begin{align}
MMSE\left( \MAULTE\right) &= Cov\left( \MAULTE\right)+Bias\left( \MAULTE\right) Bias\left(\AULE\right)^{\top} \nonumber\\
&=\bM \bV^{-1}\bM ^{\top}+\left( k+d\right)^{2}\bU\bVk^{-1}\bbeta\bbeta^{\top}\bVk^{-1}\bU^{\top}
\end{align}
where $\bU=\bI+\left( d+k\right)\bVk^{-1}-\left( d+k\right)^{2} \bVk^{-2}$ and
\begin{eqnarray}
MSE\left( \MAULTE\right)  &=& tr\left( Cov\left( \MAULTE\right)\right)
+Bias\left( \bhat\right)^{\top} Bias\left( \bhat\right)\nonumber \\
&=&\sum\limits_{j=1}^{p}\frac{1}{\lambda_j} \ll 1- \frac{(k+d)}{\ll \lambda _{j}+k \rr} - \frac{(k+d)^2}{\ll \lambda _{j}+k \rr^2} + \frac{(k+d)^3}{\ll \lambda _{j}+k \rr^3}\rr^2 
\no\\
&&+ (k+d)^2\sum\limits_{j=1}^{p}\frac{\alpha_j^2}{\ll \lambda _{j}+k \rr^2}\ll1+ \frac{(k+d)}{\ll \lambda _{j}+k \rr} - \frac{(k+d)^2}{\ll \lambda _{j}+k \rr^2}\rr^2\no\\
&=&\sum\limits_{j=1}^{p}\frac{\left(\lambda _{j}-d\right) ^{4}\left( \lambda _{j}+d+2k\right) ^{2}}{\lambda_{j}\left( \lambda _{j}+k\right) ^{6}} \nonumber\\
&&+\left( k+d\right)
^{2}\sum\limits_{j=1}^{p}\frac{\alpha _{j}^{2}\left(\left(
\lambda _{j}+k\right) ^{2}+\left( k+d\right)\left( \lambda _{j}-d\right) \right)^{2}}{\left( \lambda _{j}+k\right) ^{6}}.
\label{mse:MAULTE}
\end{eqnarray}


\section{Theoretical Comparisons Between Estimators}\label{sec:comp}

This section presents some theorems about the superiority of AULTE over LTE and MLE. The squared bias of an estimator $\bhat$ is described as follows:
\begin{eqnarray*}
SB\left( \bhat\right) = Bias\left( \bhat\right)^{\top} Bias\left( \bhat\right)=\Big\Vert  Bias\left( \bhat\right) \Big\Vert _2^2.
\end{eqnarray*}
Thus, we compare the squares biases of LTE and AULTE in the following theorem.

\begin{theorem}\label{sb:LTE:AULTE}
If $\left( \lambda
_{j}-d\right) \left( \lambda _{j}+d+2k\right) >0$ for all \(  j = 1, 2, \ldots, p \), then the squared bias of AULTE is lower than that of LTE  namely,
\begin{eqnarray*}
\Big\Vert Bias\left(\LTE\right) \Big\Vert _2^2-\Big\Vert Bias\left(\AULTE\right) \Big\Vert _2^2>0.
\end{eqnarray*}
\end{theorem}
\begin{proof}
   The difference in squared bias is:

\begin{eqnarray}
\left\Vert Bias\left( \LTE\right) \right\Vert ^{2}-\left\Vert
Bias\left( \AULTE\right) \right\Vert ^{2}
&=&\sum\limits_{j=1}^{p}%
\frac{\left( k+d\right) ^{2}\alpha _{j}^{2}}{\left( \lambda _{j}+k\right)
^{2}}-\sum\limits_{j=1}^{p}\frac{\left( k+d\right) ^{4}\alpha _{j}^{2}}{%
\left( \lambda _{j}+k\right) ^{4}} \nonumber\\
&=&\left( k+d\right) ^{2}\sum\limits_{j=1}^{p}\alpha _{j}^{2}\frac{\left(
\lambda _{j}+k\right) ^{2}-\left( k+d\right) ^{2}}{\left( \lambda
_{j}+k\right) ^{4}}
\end{eqnarray}
Considering that $\left( k+d\right) ^{2},\alpha _{j}^{2}$ and $\left(\lambda _{j}+k\right)^{4}$ are all positive, it is sufficient for $\left\Vert Bias\left( \LTE\right) \right\Vert ^{2}-\left\Vert Bias\left( \AULTE\right)
\right\Vert ^{2}$ to be positive when $\left( \lambda _{j}+k\right) ^{2}-\left( k+d\right) ^{2}=\left( \lambda
_{j}-d\right) \left( \lambda _{j}+d+2k\right) >0$ for all \(  j = 1, 2, \ldots, p \). Thus, the proof is completed.
\end{proof}
Now, we compare the squared biases of LTE and MAULTE in the following theorem.

\begin{theorem}\label{sb:LTE:MAULTE}
If $d>2\lambda _{j}+k$ or $-k<d<\lambda _{j}$ or $d<-\lambda _{j}-2k$ for all $j=1,2,\ldots,p$ then the squared bias of MAULTE is lower than that of LTE, namely,
\begin{eqnarray*}
\Big\Vert Bias\left(\LTE\right) \Big\Vert _2^2-\Big\Vert Bias\left(\MAULTE\right) \Big\Vert _2^2>0.
\end{eqnarray*}
\end{theorem}
\begin{proof}
   The difference in squared bias is:
\begin{align}
\label{bias: maulte:lte}
&\left\Vert Bias\left( \LTE\right) \right\Vert ^{2}-\left\Vert 
Bias\left( \MAULTE\right) \right\Vert ^{2} \nonumber\\
& =\sum_{j=1}^{p}\frac{(k+d)^{2}\alpha _{j}^{2}}{%
(\lambda _{j}+k)^{2}} -\sum_{j=1}^{p}\frac{(k+d)^{2}\alpha _{j}^{2}((\lambda
_{j}+k)^{2}+(k+d)(\lambda _{j}-d))^{2}}{(\lambda _{j}+k)^{6}} \nonumber\\
& =\sum_{j=1}^{p}\frac{(k+d)^{2}\alpha _{j}^{2}}{(\lambda _{j}+k)^{2}} \left( \frac{2(d-\lambda _{j})(k+d)(\lambda
_{j}+k)^{2}-(k+d)^{2}(\lambda _{j}-d)^{2}}{(\lambda _{j}+k)^{4}}\right) 
\end{align}

To evaluate the positivity of Eq. (\ref{bias: maulte:lte}), we examined the roots of function defined in Eq. (\ref{g: bias: maulte:lte}).
\begin{equation}
\label{g: bias: maulte:lte}
f_{j}(d)=(\lambda _{j}-d)(k+d)(d^{2}-d(\lambda _{j}-k)-2\lambda
_{j}^{2}-5k\lambda _{j}-2k^{2})
\end{equation}
If the function $f(.)$ is positive, the squared bias difference will be positive. When the root analysis of the function $f(.)$ is performed, it is seen that the squared bias difference $\left\Vert Bias\left( \LTE\right) \right\Vert ^{2}-\left\Vert Bias\left(\MAULTE\right)\right\Vert ^{2}$ is positive if $d>2\lambda _{j}+k$ or $-k<d<\lambda _{j}$ or $d<-\lambda _{j}-2k$. Thus, the proof is completed.
\end{proof}
Then, we compare the squared biases of AULTE and MAULTE in the following theorem.

\begin{theorem}\label{sb:AULTE:MAULTE}
If \( d < -2k - \lambda_j \), \( \lambda_j - \sqrt{2(\lambda_j + k)} < d < \lambda_j \), or \( d > \lambda_j + \sqrt{2(\lambda_j + k)} \) for all \(  j = 1, 2, \ldots, p \) for all \(  j = 1, 2, \ldots, p \) then the squared bias of MAULTE is smaller than that of AULTE, namely,
\begin{eqnarray*}
\Big\Vert Bias\left(\AULTE\right) \Big\Vert _2^2-\Big\Vert Bias\left(\MAULTE\right) \Big\Vert _2^2>0.
\end{eqnarray*}
\end{theorem}
\begin{proof}
   The difference in squared bias is:
\begin{align}
&\left\Vert Bias\left( \AULTE\right) \right\Vert ^{2}-\left\Vert Bias\left( %
\MAULTE\right) \right\Vert ^{2}& \notag \\ 
& =(k+d)^{4}\sum_{j=1}^{p}\frac{\alpha
_{j}^{2}}{(\lambda _{j}+k)^{4}}  \notag  \label{bias: maulte:aulte}  -(k+d)^{2}\sum_{j=1}^{p}\frac{\alpha _{j}^{2}((\lambda
_{j}+k)^{2}+(k+d)(\lambda _{j}-d))^{2}}{(\lambda _{j}+k)^{6}}  \notag \\
& =\sum_{j=1}^{p}\frac{(k+d)^{2}\alpha _{j}^{2}}{(\lambda _{j}+k)^{6}}\left(
\left( k+d\right) +\left( \lambda _{j}+k\right) ^{2}(2\lambda
_{j}+k-d)\right)  \left( -\left( \lambda _{j}+k\right) ^{2}+\left( k+d\right)
^{2}\right)   \notag \\
& =\sum_{j=1}^{p}\frac{(k+d)^{2}\alpha _{j}^{2}}{(\lambda _{j}+k)^{6}}  (\lambda _{j}-d)(2k+d+\lambda _{j})\left( d^{2}-2d\lambda
_{j}-2k^{2}-4k\lambda _{j}-\lambda _{j}^{2}\right). 
\end{align}

To evaluate the positivity of Eq. \eqref{bias: maulte:aulte}, we examined the roots of function defined in Eq. (\ref{g: bias: maulte:aulte}).
\begin{equation}
\label{g: bias: maulte:aulte}
g_j(d) = (d^2 - 2d \lambda_j - 2k^2 - 4k \lambda_j - \lambda_j^2)(\lambda_j - d)(2k + d + \lambda_j).
\end{equation}
If the function $g(.)$ is positive, the squared bias difference will be positive. When the root analysis of the function $g(.)$ is performed, it is seen that the squared bias difference $\left\Vert Bias\left( \AULTE\right) \right\Vert ^{2}-\left\Vert Bias\left(\MAULTE\right)\right\Vert ^{2}$ is positive if \( d < -2k - \lambda_j \) or \( \lambda_j - \sqrt{2(\lambda_j + k)} < d < \lambda_j \), or \( d > \lambda_j + \sqrt{2(\lambda_j + k)} \). Thus, the proof is completed.
\end{proof}


Then, we compare the variances of LTE and AULTE in the following theorem.

\begin{theorem} \label{var:LTE:AULTE}
In the Bell regression model, if  $-\left( k+d\right) \left( 2\lambda _{j}+3k+d\right) >0$ for all \(  j = 1, 2, \ldots, p \) then the AULTE has a lower variance than LTE, namely,
\begin{eqnarray*}
 Var\left(\LTE\right) -Var\left(\AULTE\right)>0. 
\end{eqnarray*}
\end{theorem}
\begin{proof}
The difference in variances between LTE and AULTE is computed as 
\begin{eqnarray*}
Var\left( \LTE\right) -Var\left( \AULTE\right) 
&=&\sum\limits_{j=1}^{p}\frac{%
\left( \lambda _{j}-d\right) ^{2}}{\lambda _{j}\left( \lambda _{j}+k\right)
^{2}}-\sum\limits_{j=1}^{p}\frac{\left( \lambda _{j}-d\right) ^{2}\left(
\lambda _{j}+2k+d\right) ^{2}}{\lambda _{j}\left( \lambda _{j}+k\right) ^{4}}
\\
&=&\sum\limits_{j=1}^{p}\frac{\left( \lambda _{j}-d\right) ^{2}}{\lambda
_{j}\left( \lambda _{j}+k\right) ^{4}}\left( \left( \lambda _{j}+k\right)
^{2}-\left( \lambda _{j}+2k+d\right) ^{2}\right)  \\
&=&\sum\limits_{j=1}^{p}\frac{\left( \lambda _{j}-d\right) ^{2}}{\lambda
_{j}\left( \lambda _{j}+k\right) ^{4}}\left( -\left( k+d\right) \left(
2\lambda _{j}+3k+d\right) \right). 
\end{eqnarray*}
Since we have $\left( \lambda _{j}-d\right)^2>0$, $\lambda _{j}>0$ and $\left( \lambda _{j}+k\right)^{4}>0$, if $-\left( k+d\right) \left( 2\lambda _{j}+3k+d\right) >0$ for all $j=1,2,\ldots,p$, then the difference between the variances of LTE and AULTE is positive. Thus, the proof is completed.
\end{proof}

We use the following lemma to compare MMSE functions of estimators. The following lemma will be employed to theoretically demonstrate situations where one estimator yields better results than another in terms of MMSE.

\begin{Lemma}
\label{lemma: 3.1}
(\cite{trenkler1990}) \\
Let \( \hat{\bbeta}_1 \) and \( \hat{\bbeta}_2 \) be two estimators of \( \beta \). Let \( \bD = \operatorname{Cov}(\hat{\bbeta}_1) - \operatorname{Cov}(\hat{\bbeta}_2) \) be a positive definite matrix, and let \( \ba_1 = \operatorname{bias}(\hat{\bbeta}_1) \) and \( \ba_2 = \operatorname{bias}(\hat{\bbeta}_2) \). Then,
$\ba_2^\top (\bD + \ba_1 \ba_1^\top)^{-1} \ba_2 < 1$
if and only if \( \operatorname{MMSE}(\hat{\bbeta}_1) - \operatorname{MMSE}(\hat{\bbeta}_2) \) is a positive definite matrix.
\end{Lemma}
According to Lemma \ref{lemma: 3.1}, if the MMSE difference of two estimators $\bhat_1$ and $\bhat_2$ is a positive definite matrix, i.e. $MMSE\ll \bhat_1 \rr-MMSE\ll \bhat_2 \rr > 0$, as given in that lemma then it can be concluded that $\bhat_2$ is superior to $\bhat_1$ in terms of MMSEs.

In the following theorem, we compare the MMSE of AULTE with the MMSE of LTE for Bell regression model.
\begin{theorem} \label{MMSE:LTE:AULTE}
Let \( -(k + d)(2\lambda_j + 3k + d) > 0 \) for all \(  j = 1, 2, \ldots, p \). Then AULTE is superior to LTE based on the MMSE criterion, i.e., $$MMSE\ll \lte \rr-MMSE\ll \AULTE \rr > 0$$  if and only if
\[
\bb_{AULTE}^\top (\bS_1 + \bb_{LTE} \bb_{LTE}^\top)^{-1} \bb_{AULTE} < 1
\]
where \( \bS_1 = \operatorname{Cov}(\LTE) - \operatorname{Cov}(\AULTE) \), and \( \bb_{LTE} \) and \( \bb_{AULTE} \) represent the bias vectors of LTE and AULTE, which are given in Eqs. (\ref{bias:LTE}) and (\ref{bias: aulte}), respectively.
\end{theorem}
\begin{proof}
Using Eqs. \eqref{cov:LTE} and \eqref{cov:AULTE}, the difference between the covariances of LTE and AULTE is given as follows:
\[
\begin{aligned}
\bS_1 
&= \operatorname{Cov}(\LTE) - \operatorname{Cov}(\AULTE) \\
&= \left( \bVk^{-1} \bVd \bV^{-1} \bVd \bVk^{-1} 
- \left( \bI - (k + d)^2 \bVk^{-2} \right) \bV^{-1} 
\left( \bI - (k + d)^2 \bVk^{-2} \right) \right) \\
&= \bQ^\top \operatorname{diag} \left( 
\frac{(\lambda_j - d)^2}{\lambda_j (\lambda_j + k)^2}
- 
\frac{(\lambda_j - d)^2 (\lambda_j + 2k + d)^2}{\lambda_j (\lambda_j + k)^4}
\right) \bQ \\
&= \bQ^\top \operatorname{diag} \left( 
\frac{(\lambda_j - d)^2}{\lambda_j (\lambda_j + k)^4} 
\left( - (k + d)(2\lambda_j + 3k + d) \right)
\right) \bQ
\end{aligned}
\]
where \( \bVk^{-1} = \operatorname{diag} \left( \frac{1}{\lambda_j + k} \right) \) and \( \bVd = \operatorname{diag}(\lambda_j - d) \). It is evident that \( \bS_1 \) is positive definite if \( -(k + d)(2\lambda_j + 3k + d) > 0 \) for all \( j = 1, 2, \ldots, p \). Thus, using Lemma (\ref{lemma: 3.1}), $$MMSE\ll \lte \rr-MMSE\ll \AULTE \rr > 0$$ if and only if
\[
\bb_{AULTE}^\top (\bS_1 + \bb_{LTE} \bb_{LTE}^\top)^{-1} \bb_{AULTE} < 1.
\]
This completes the proof.
\end{proof}

In the following theorem, we compare the MMSE of MAULTE with the MMSE of LTE for Bell regression model.
\begin{theorem}\label{MMSE:LTE:MAULTE}
Let \( k - \sqrt{2(\lambda_j + k)} < d < k + \sqrt{2(\lambda_j + k)} \) for all \(  j = 1, 2, \ldots, p \), Then MAULTE is superior to LTE based on the MMSE criterion, i.e., $$MMSE\ll \lte \rr-MMSE\ll \MAULTE \rr > 0$$ if and only if
\[
\bb_{MAULTE}^\top (\bS_2 + \bb_{LTE} \bb_{LTE}^\top)^{-1} \bb_{MAULTE} < 1
\]
where \( \bS_2 = \operatorname{Cov}(\LTE) - \operatorname{Cov}(\MAULTE) \), and \( \bb_{LTE} \) and \( \bb_{MAULTE} \) represent the bias vectors of LTE and MAULTE, which are given in Eqs. (\ref{bias:LTE}) and (\ref{bias:maulte}), respectively.
\end{theorem}

\begin{proof}
The difference between the covariances of LTE and MAULTE is given as
follows: 
\begin{eqnarray*}
\bS_{2} &=&\operatorname{Cov}(\LTE)-\operatorname{Cov}(\MAULTE) \\
&=&\bVk^{-1}\bVd\bV^{-1}\bVd\bVk^{-1}\\&&-
\left( \bI-(k+d)^{2}\bVk^{-2}\right) \left( \bI-(k+d)\bVk^{-1}\right) \bV^{-1}\left( \bI-(k+d)\bVk%
^{-1}\right) \left( \bI-(k+d)^{2}\bVk^{-2}\right)  \\
&=&\bQ^{\top }\operatorname{diag}\left( \frac{(\lambda _{j}-d)^{2}}{\lambda
_{j}(\lambda _{j}+k)^{2}}-\frac{(\lambda _{j}-d)^{4}(\lambda _{j}+2k+d)^{2}}{%
\lambda _{j}(\lambda _{j}+k)^{6}}\right) \bQ \\
&=&\bQ^{\top }\operatorname{diag}\left( \frac{(\lambda _{j}-d)^{2}}{\lambda
_{j}(\lambda _{j}+k)^{6}}\left( -d^{2}+d(-2k)+\left( 2\lambda
_{j}^{2}+4k\lambda _{j}+k^{2}\right) \right) \right) \bQ.
\end{eqnarray*}%
Now, let $f_j(d)= -d^{2}+d(-2k)+\left( 2\lambda
_{j}^{2}+4k\lambda _{j}+k^{2}\right)$ for all 
$ j=1, 2, \ldots, p$. Discriminant analysis of $f_j(d)$ shows that if $k-\sqrt{2(\lambda _{j}+k)}<d<k+\sqrt{2(\lambda _{j}+k)}$ then $f_j(d)>0$ so that $\bS_2$ becomes positive definite. Similar to Theorem \ref{MMSE:LTE:AULTE}, by Lemma (\ref{lemma: 3.1}),  
$$MMSE\ll \lte \rr-MMSE\ll \MAULTE \rr > 0$$ if and only if
\[
\bb_{MAULTE}^\top (\bS_2 + \bb_{LTE} \bb_{LTE}^\top)^{-1} \bb_{MAULTE} < 1
\]
which completes the proof.
\end{proof}

Now, in the following theorem, we compare the MMSE of AULTE with the MMSE of MAULTE for Bell regression model.
\begin{theorem}\label{mmse:AULTE:MAULTE}
Let \( (k + d)(2\lambda_j + k - d) > 0 \) for all \( j = 1, 2, \ldots, p \). Then MAULTE is superior to  AULTE based on the MMSE criterion, i.e.,
$$MMSE\ll \AULTE \rr-MMSE\ll \MAULTE \rr > 0$$ 
if and only if
\[
\bb_{MAULTE}^\top (\bS_3 + \bb_{AULTE} \bb_{AULTE}^\top)^{-1} \bb_{MAULTE} < 1
\]
where \( \bS_3 = \operatorname{Cov}(\AULTE) - \operatorname{Cov}(\MAULTE) \), and \( \bb_{AULTE} \) and \( \bb_{MAULTE} \) represent the bias vectors of AULTE and MAULTE which are given in Eqns. \eqref{bias: aulte} and \eqref{bias:maulte} respectively.
\end{theorem}

\begin{proof}
The difference between the covariances of AULTE and MAULTE is given as follows:
\begin{eqnarray*}
\bS_{3} &=&\operatorname{Cov}(\AULTE)-\operatorname{Cov}(\MAULTE) \\
&=&\left( \bI-(k+d)^{2}\bVk^{-2}\right) \bV^{-1}\left( \bI-(k+d)^{2}\bVk%
^{-2}\right) -\left( \bI-(k+d)^{2}\bVk^{-2}\right)  \\
&=&\left( \bI-(k+d)^{2}\bVk^{-2}\right) \bV^{-1}\left( \bI-(k+d)^{2}\bVk%
^{-2}\right)\\&& -\left( \bI-(k+d)^{2}\bVk^{-2}\right) 
 \left( \bI-(k+d)\bVk^{-1}\right) \bV^{-1}\left( \bI-(k+d)\bVk%
^{-1}\right) \left( \bI-(k+d)^{2}\bVk^{-2}\right)  \\
&=&\bQ^{\top }\operatorname{diag}\left( \frac{(\lambda _{j}-d)^{2}(\lambda
_{j}+2k+d)^{2}}{\lambda _{j}(\lambda _{j}+k)^{4}}-\frac{(\lambda
_{j}-d)^{4}(\lambda _{j}+2k+d)^{2}}{\lambda _{j}(\lambda _{j}+k)^{6}}\right) %
\bQ \\
&=&\bQ^{\top }\operatorname{diag}\left( \frac{(\lambda _{j}-d)^{2}(\lambda
_{j}+2k+d)^{2}}{\lambda _{j}(\lambda _{j}+k)^{6}}(k+d)(2\lambda
_{j}+k-d)\right) \bQ.
\end{eqnarray*}
If \( (k + d)(2\lambda_j + k - d) > 0 \)
for all \(j = 1, 2, \ldots, p \), then \( \bS_3 \) becomes positive definite. Thus, similar to Theorems \ref{MMSE:LTE:AULTE} and \ref{MMSE:LTE:MAULTE}, using Lemma (\ref{lemma: 3.1}),
$$MMSE\ll \AULTE \rr-MMSE\ll \MAULTE \rr > 0$$ 
if and only if
\[
\bb_{MAULTE}^\top (\bS_3 + \bb_{AULTE} \bb_{AULTE}^\top)^{-1} \bb_{MAULTE} < 1
\]
which completes the proof.
\end{proof}
\subsection{Selection of the biasing parameter d}\label{d_est}

In this subsection, we focus on solving the following equation to estimate the biasing parameter $d$ of the AULTE estimator in Bell regression by taking the partial derivative of MSE given in Equation \eqref{eqMSE} with respect to $d$.
$d_{opt}$ is obtained by taking the first derivative of the $MSE\left(\lte\right) $ with respect to $d$ and setting the equation equal to zero.  
\begin{eqnarray*}
\frac{\partial MSE\left( \AULTE\right) }{\partial d} &=&\frac{-2\left( \lambda _{j}-d\right) \left( \lambda _{j}+d+2k\right)
^{2}+2\left( \lambda _{j}+d+2k\right) \left( \lambda _{j}-d\right) ^{2}}{%
\lambda _{j}\left( \lambda _{j}+k\right) ^{4}} \\
&&+4\left( k+d\right) ^{3}\sum\limits_{j=1}^{p}\frac{\alpha _{j}^{2}}{\left(
\lambda _{j}+k\right) ^{4}} \\
&=&2\sum\limits_{j=1}^{p}\frac{\left( \lambda _{j}-d\right) \left( \lambda
_{j}+d+2k\right) }{\lambda _{j}\left( \lambda _{j}+k\right) ^{4}} \left( -2\left( k+d\right) +4\left( k+d\right) ^{3}\right)
\sum\limits_{j=1}^{p}\frac{\alpha _{j}^{2}}{\left( \lambda _{j}+k\right) ^{4}%
} \\
&=&-4\left( k+d\right) \sum \frac{\left( \lambda _{j}-d\right) \left(
\lambda _{j}+d+2k\right) }{\lambda _{j}\left( \lambda _{j}+k\right) ^{4}} +4\left( k+d\right) ^{3}\sum\limits_{j=1}^{p}\frac{\alpha _{j}^{2}}{\left(
\lambda _{j}+k\right) ^{4}} \\
&=&\sum\limits_{j=1}^{p}\frac{\left( k+d\right) }{\lambda _{j}\left( \lambda
_{j}+k\right) ^{4}}\left( -\left( \lambda _{j}-d\right) \left( \lambda
_{j}+2k+d\right) +\left( k+d\right) ^{2}\lambda _{j}\alpha _{j}^{2}\right)\\
&=&0
\end{eqnarray*}
Then, we obtain $d_{opt}$ as follows:
\begin{align}\label{dopt}
    d_{opt} = median \left( -k+\frac{\left( \lambda _{j}+k\right) \sqrt{1+\lambda
_{j}\alpha _{j}^{2}}}{\left( 1+\lambda _{j}\alpha _{j}^{2}\right) }\right).
\end{align}
In order to use $d_{opt}$ as an estimator of $d$, one needs a value for $k$ as well. In this regard, we suggest using a suitable estimator of $k$ from the literature, as we do in the simulation section.
\section{Monte Carlo Simulation}\label{sec:sim} 
This section provides a comprehensive Monte Carlo simulation study to compare the MSEs and SBs of the examined estimators in this paper. One of our objectives is to assess the performance of the mentioned estimators when there is multicollinearity. We generate the design matrix $\bX$ using a multivariate normal distribution having zero mean vector and a specific covariance matrix such that the pairwise correlations between the variables $\bx_i$ and $\bx_j$ is set to be $\rho^{\vert i-j\vert}$ following \citet{enet}. 

The simulation settings are listed as follows:
\begin{itemize}
\item The number of repetitions is $1000$.
\item The sample sizes, $n=100, 200, 400$.
\item The number of predictor variables, $p=4, 8, 12$
\item The degree of correlation between the predictors, $\rho=0.90, 0.95, 0.99$. 
\item The $n$ observations of the response variable are generated from $y_i \sim Bell\left(\mu_i\right)$ where 
\begin{equation}    \mu_i=\exp(\bx_i^\top\bbeta),~i=1,2,\ldots,n\nonumber.
\end{equation} 
\item We estimated parameter $k=1/(\ahat^{\top}\ahat)$ for LTE where $\ahat = \bQ^{\top}\mle$.
\item We estimated the biasing parameters $k$ and $d$ for AULTE and MAULTE as follows: Firstly, the parameter $k$ is computed by $k = 1/(min(\ah_j^2))$. Afterwards, we use \textcolor{blue}{\texttt{optim}} function in \textcolor{blue}{\texttt{R}} to find the optimal value of the parameter $d$ for $MSE\left( \AULTE\right)$ and $MSE\left( \MAULTE\right)$ using an initial value computed via Eq.(\ref{dopt}).
\item We estimated parameter $d$ for LTE following \citet{isiler2022} as follows:
\begin{align}
    d_{LTE} = \frac{\sum_{j=1}^{p}\frac{1-k\ah_j^2 }{\left(\lambda_j+k \right)^2}}{\sum_{j=1}^{p}\frac{1+\lambda_j\ah_j^2 }{\lambda_j\left(\lambda_j+k \right)^2}}.
\end{align}
\item We compute the simulated MSE and squared bias of any estimator $\hbbeta^{\ast}$ respectively as follows:
\begin{eqnarray*}
{\rm MSE}\left( \hbbeta^{\ast} \right)&=&\frac{1}{1000}\sum_{r=1}^{1000}\left( \hbbeta^{\ast} -\bbeta\right)_r^{\top}\left(\hbbeta^{\ast} -\bbeta\right),\\
{\rm SB}\left( \hbbeta^{\ast} \right)&=&\frac{1}{1000}\sum_{r=1}^{1000}\left( E\ll\hbbeta^{\ast}\rr -\bbeta\right)_r^{\top}\left( E\ll\hbbeta^{\ast}\rr -\bbeta\right).
\end{eqnarray*}
\end{itemize}
The simulation results are given in Tables \ref{Table1}-\ref{Table4}.

\begin{table}[H]
\caption{Simulated MSE values and their standard errors for $p=4$}
\centering

\begin{tabular}{ccrrrr|rrrr}
  \multicolumn{2}{c}{}&\multicolumn{4}{c}{MSE}&\multicolumn{4}{c}{SE}\\ 
  \cline{3-6}\cline{7-10}
  $\rho$ & $n$ & MLE & LT & AULT & MAULT & MLE & LT & AULT & MAULT \\ 
  \hline
  0.90 & 100 & 4.9991 & 4.5999 & 4.3721 & 4.2725 & 2.4935 & 1.6136 & 1.6131 & 1.3925 \\ 
  0.90 & 200 & 4.2448 & 4.1053 & 4.0696 & 4.0227 & 1.4548 & 0.9460 & 0.8348 & 0.3815 \\ 
  0.90 & 400 & 4.0826 & 4.0500 & 4.0080 & 3.9993 & 0.9288 & 0.6064 & 0.5755 & 0.3532 \\ \hline
  0.95 & 100 & 5.6233 & 4.8185 & 4.6107 & 4.2906 & 3.4631 & 2.3091 & 2.2437 & 1.5910 \\ 
  0.95 & 200 & 4.6139 & 4.2978 & 4.2346 & 4.0867 & 2.0729 & 1.3239 & 1.2364 & 0.6636 \\ 
  0.95 & 400 & 4.3142 & 4.1471 & 4.0985 & 3.9963 & 1.4819 & 0.9372 & 0.8833 & 0.5387 \\ \hline
  0.99 & 100 & 9.1929 & 6.1519 & 5.7724 & 4.5359 & 7.7473 & 4.9007 & 4.6299 & 2.8362 \\ 
  0.99 & 200 & 7.0334 & 5.3250 & 5.0685 & 4.3342 & 5.1924 & 3.3371 & 3.0272 & 2.2992 \\ 
  0.99 & 400 & 5.1894 & 4.4800 & 4.3656 & 4.0824 & 3.2172 & 2.0273 & 1.8908 & 0.9856 \\  \hline
\end{tabular}

\label{Table1}
\end{table}

\begin{table}[H]
\caption{Simulated MSE values and their standard errors for $p=8$}
\centering

\begin{tabular}{ccrrrr|rrrr}
  \multicolumn{2}{c}{}&\multicolumn{4}{c}{MSE}&\multicolumn{4}{c}{SE}\\ 
  \cline{3-6}\cline{7-10}
$\rho$ & $n$ & MLE & LT & AULT & MAULT & MLE & LT & AULT & MAULT \\ 
  \hline
  0.90 & 100 & 9.4780 & 8.6109 & 8.3996 & 8.0643 & 3.4145 & 2.0355 & 1.8520 & 0.6454 \\ 
  0.90 & 200 & 8.7100 & 8.2806 & 8.1693 & 8.0222 & 2.2223 & 1.3119 & 1.1561 & 0.4348 \\ 
  0.90 & 400 & 8.2691 & 8.1123 & 8.0579 & 8.0210 & 1.4182 & 0.7692 & 0.7062 & 0.3892 \\ \hline
  0.95 & 100 & 10.7972 & 9.1288 & 8.7817 & 8.1193 & 5.1059 & 3.0682 & 2.7955 & 1.0611 \\ 
  0.95 & 200 & 9.3269 & 8.5685 & 8.3913 & 8.0810 & 3.3743 & 2.0259 & 1.8047 & 0.7220 \\ 
  0.95 & 400 & 8.5543 & 8.3087 & 8.1689 & 8.0126 & 1.8003 & 1.0235 & 0.9399 & 0.3967 \\ \hline
  0.99 & 100 & 19.3679 & 12.4639 & 11.3568 & 8.8511 & 11.8778 & 6.9932 & 6.4396 & 3.5137 \\ 
  0.99 & 200 & 15.0639 & 10.8081 & 10.0774 & 8.5884 & 9.4022 & 5.7724 & 5.4119 & 2.9262 \\ 
  0.99 & 400 & 11.2247 & 9.2931 & 8.9072 & 8.1999 & 5.6893 & 3.4332 & 3.2618 & 1.8109 \\ \hline
\end{tabular}

\label{Table2}
\end{table}

\begin{table}[H]
\caption{Simulated MSE values and their standard errors for $p=12$}
\centering

\begin{tabular}{ccrrrr|rrrr}
  \multicolumn{2}{c}{}&\multicolumn{4}{c}{MSE}&\multicolumn{4}{c}{SE}\\ 
  \cline{3-6}\cline{7-10}
  $\rho$ & $n$ & MLE & LT & AULT & MAULT & MLE & LT & AULT & MAULT \\ 
  \hline
  0.90 & 100 & 13.8523 & 13.2769 & 12.6514 & 12.0587 & 2.6052 & 1.5476 & 1.5272 & 0.5618 \\ 
  0.90 & 200 & 12.6684 & 12.2569 & 12.1105 & 12.0048 & 2.4129 & 1.2734 & 1.1259 & 0.3606 \\ 
  0.90 & 400 & 12.4097 & 12.3609 & 12.0364 & 11.9833 & 1.3711 & 0.6812 & 0.7133 & 0.2269 \\ \hline
  0.95 & 100 & 17.5391 & 14.5056 & 13.6797 & 12.2273 & 8.1189 & 4.7443 & 4.4961 & 1.4415 \\ 
  0.95 & 200 & 13.5470 & 12.5474 & 12.3418 & 12.0781 & 3.4399 & 1.8785 & 1.7144 & 0.9369 \\ 
  0.95 & 400 & 12.7917 & 12.3276 & 12.2044 & 12.0397 & 2.5129 & 1.3337 & 1.2385 & 0.5720 \\ \hline
  0.99 & 100 & 42.7217 & 25.4069 & 21.0606 & 13.8814 & 26.4427 & 16.4151 & 15.1303 & 11.2931 \\ 
  0.99 & 200 & 21.7566 & 15.6264 & 14.4257 & 12.2480 & 10.2769 & 5.7533 & 5.3620 & 1.8131 \\ 
  0.99 & 400 & 14.8900 & 13.2926 & 12.6780 & 12.0762 & 4.4479 & 2.2041 & 2.1920 & 0.6723 \\ \hline
\end{tabular}

\label{Table3}
\end{table}

\begin{table}[H]
\caption{Simulated SB values}
\label{Table4}
\centering

\begin{tabular}{cc|rrr|rrr|rrr}
  \multicolumn{2}{c}{}&\multicolumn{3}{c}{$p=4$}&\multicolumn{3}{c}{$p=8$}&\multicolumn{3}{c}{$p=12$}\\ 
  $\rho$ & $n$ & LT & AULT & MAULT & LT & AULT & MAULT & LT & AULT & MAULT \\ 
  \hline
  0.90 & 100 & 4.2228 & 4.0348 & 4.0545 & 8.0272 & 8.0052 & 8.0141 & 12.8682 & 12.2952 & 11.9972 \\ 
  0.90 & 200 & 3.9715 & 3.9767 & 4.0040 & 7.9946 & 7.9990 & 7.9910 & 11.9659 & 11.9379 & 11.9844 \\ 
  0.90 & 400 & 4.0008 & 3.9660 & 3.9792 & 8.0149 & 7.9969 & 8.0023 & 12.2759 & 11.9520 & 11.9633 \\ \hline
  0.95 & 100 & 4.1759 & 4.0830 & 4.0258 & 7.9846 & 7.9803 & 7.9804 & 12.4347 & 12.3272 & 12.0935 \\ 
  0.95 & 200 & 4.0469 & 4.0381 & 4.0469 & 8.0286 & 8.0441 & 8.0392 & 11.9752 & 11.9855 & 12.0165 \\ 
  0.95 & 400 & 4.0094 & 3.9879 & 3.9479 & 8.1367 & 8.0492 & 7.9840 & 12.0248 & 12.0163 & 12.0074 \\ \hline
  0.99 & 100 & 4.1604 & 4.1483 & 4.0676 & 7.9291 & 7.9643 & 8.0138 & 12.5470 & 12.1531 & 11.9822 \\ 
  0.99 & 200 & 4.1291 & 4.1169 & 4.0747 & 8.0078 & 7.9985 & 8.0430 & 12.0664 & 12.0324 & 11.9529 \\ 
  0.99 & 400 & 3.9085 & 3.9316 & 3.9878 & 7.9460 & 7.9358 & 7.9900 & 12.4801 & 12.1274 & 12.0156 \\ 
   \hline
\end{tabular}

\end{table}

From Tables \ref{Table1}- \ref{Table4}, the findings show the following inferences.

\begin {itemize}
    \item As the sample size increases, the MSEs, SBs and SEs of all estimators decrease as expected.
     \item We observe that the MSE increases when $\rho$ increase. 
     \item We concluded that MAULTE is superior to its competitors MLE, LTE, and AULTE according to MSE criterion.

\item As expected, SB values generally decrease as the sample size increases. However, we observe slight increases in some simulation settings, which may be attributed to the lack of asymptotic unbiasedness of the estimator under consideration and the influence of outliers in the sample.
     
     \item In small sample sizes, the MSE values of the three estimators show notable differences; however, as the sample size increases, these MSE values converge, becoming increasingly similar across the estimators.
    \item The performance of the two estimators proposed in this study (AULTE and MAULTE) is superior to that of LTE and MLE because they have smaller MSE and squared bias values.
    \item We conclude that MAULTE is a good choice in Bell regression for the estimation problem in the presence of multicollinearity. 
\end{itemize}

Furthermore, the results reveal that the impact of multicollinearity is not uniform between estimators. Whereas all methods are adversely affected as $\rho$ increases, the degree of deterioration varies depending on both the correlation level and the dimensionality of the problem. For example, when $p$ is large and $\rho$ is close to one, the performance gaps among the estimators become more evident, particularly in terms of MSE and squared bias. These results indicate that the MAULTE seem to provide consistent performance under such challenging conditions. This finding implies that the effectiveness of the estimators is not solely governed by sample size, but is also shaped by the structural properties of the design, such as correlation strength and dimensionality.

  In conclusion, we recommend that MAULTE is a good competitor to AULTE, LTE, and MLE in the Bell regression model.  

\section{Real Data Application}\label{sec:data}

This section focuses on a practical data set\textsuperscript{1} that includes the performance statistics of football teams in the 2014–2015 Turkish Super League season discussed by \cite{asar2018liu}. A comparable data set was previously studied by \cite{turkan2016} for the 2012–2013 season, where they illustrated the adequacy of the Poisson regression model. We define the variables as follows: The dependent variable: \textit{ the number of matches won (NWM)}. The explanatory variables include; the \textit{Number of Red Cards (NRC)}, \textit{Number of Substitutions (NS)}, \textit{Number of Matches with Over 2.5 Goals (NOG)}, \textit{Number of Matches with Goals Scored (NCG)}, and two goal-scoring ratios—\textit{NGR1} defined as the ratio of goals scored to total matches played (NGS/NM), and \textit{NGR2} defined as the ratio of goals scored to the sum of goals scored and goals conceded [NGS/(NGS + NGC)]. We consider the Bell regression model due to overdispersion in the dependent variable: the variance (9.76) exceeds the mean (7.33). The Anderson–Darling test was used to evaluate the fit of the Bell distribution to the response variable (p-value=0.3).

We observe that NWM is negatively correlated with NOG. The strongest observed correlation is between NGR1 and NCG with a coefficient of 0.76, while other pairwise correlations are relatively weak. However, as discussed by Montgomery et al.~\cite{mont2021}, high pairwise correlation is not a necessary condition for multicollinearity. 

We fit the Bell regression without intercept using maximum likelihood method, the eigenvalues of the matrix $\mathbf{X}^{\top}\mathbf{W}\mathbf{X}$ are computed as $\lambda_1 =552471.7670, \lambda_2=565.8126, \lambda_3=425.2506, \lambda_4=202.1790, \lambda_5=3.2488$,  and $\lambda_6=0.4298$ The resulting condition number being the square root of the ratio of the maximum eigenvalue to the minimum eigenvalue $\ll \sqrt{\frac{\lambda_{max}}{\lambda_{min}}}\rr$ of the matrix $\mathbf{X}^{\top}\mathbf{W}\mathbf{X}$, 1133.825, indicates the presence of severe multicollinearity within the predictor variables. Moreover, the condition indices $\ll \sqrt{\frac{\lambda_j}{\lambda_{min}}}\rr$ of the variables are computed as $1285558.9386, 1316.6019,  989.5252, 470.4548,  7.5597$ and $1.0000$.

The estimation results, including MSE, regression coefficients, are presented in Table~\ref{tab:estimation_results}.
The MSEs of the estimators MLE, LTE, AULTE and MAULTE are computed using Equations \eqref{mse:mle}, \eqref{mse:LTE}, \eqref{mse:AULTE} and \eqref{mse:MAULTE}, respectively. Since the real parameter values $\alpha_j$'s are not known, $\widehat{\alpha}_j$'s are used in the computations.

\begin{table}[H]
\caption{The estimation results of the football data set}
\label{tab:estimation_results}
\centering
\begin{tabular}{rrrrr}
  \hline
 & MLE & LTE & AULTE & MAULTE \\ 
  \hline
NRC & -0.0387 & -0.0239 & -0.0190 & -0.0054 \\ 
  NS & 0.0143 & 0.0085 & 0.0118 & 0.0063 \\ 
  NOG & -0.0878 & -0.0693 & -0.0428 & -0.0119 \\ 
  NCG & 0.0206 & 0.0544 & 0.0115 & 0.0041 \\ 
  NGR1 & 1.1645 & 0.5849 & 0.5788 & 0.1684 \\ 
  NGR2 & 0.1688 & 0.4218 & 0.0839 & 0.0244 \\ \hline
  MSEs & 2.6438 & 0.6669 & 1.0058 & 1.0755 \\ 
  SBs & 0.0000 & 0.4017 & 0.3527 & 1.0203 \\    \hline
\end{tabular}
\end{table}
\vspace{1em}
\noindent\textsuperscript{1} Please See: \url{http://www.tff.org} and \url {http://www.sahadan.com}, accessed November 15, 2016

\begin{figure}[H] 
\centering
\includegraphics[width=0.7\textwidth]{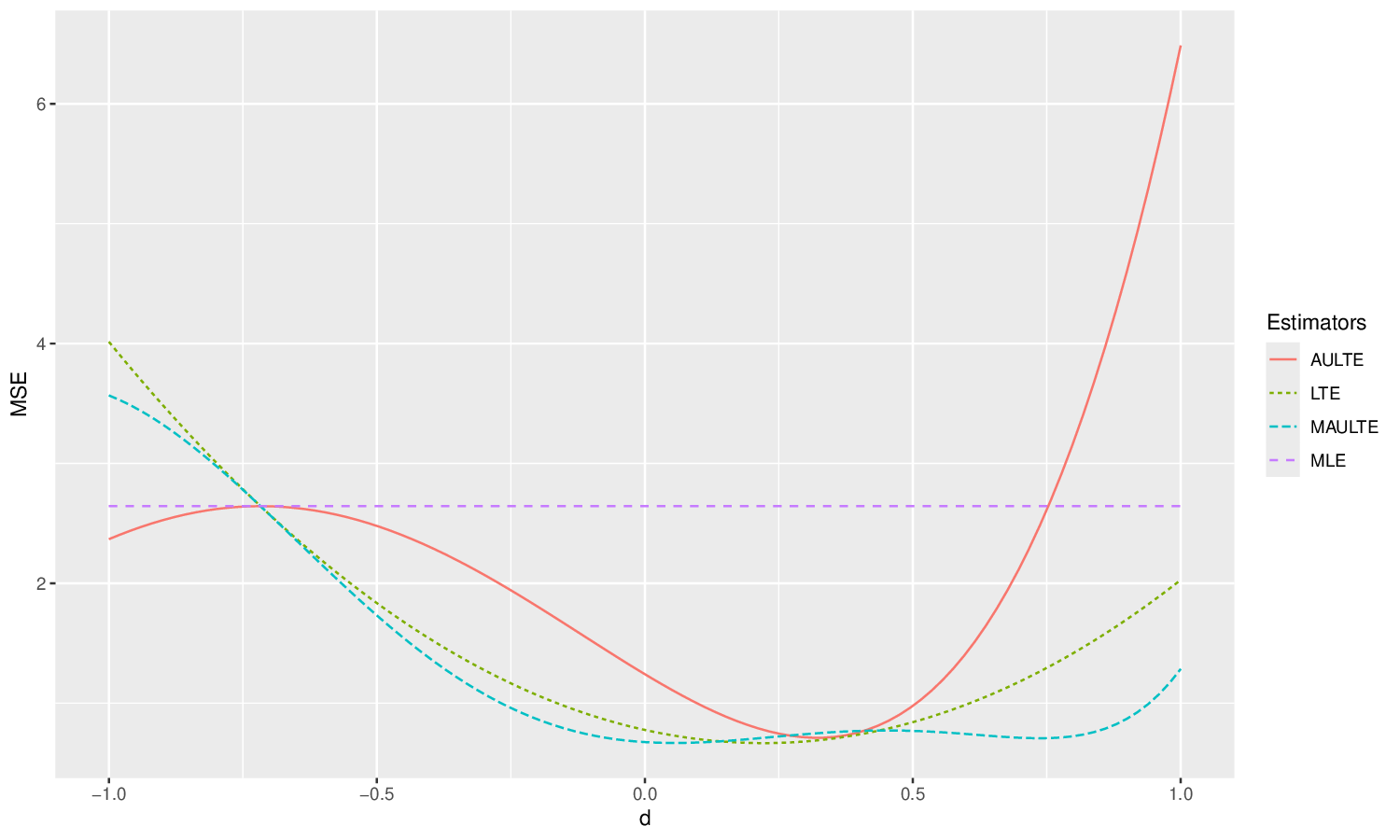}
\caption{MSE values of the estimators for a fixed $k$ and a range of $d$ values.}
\label{plot:app}
\end{figure} 
The findings of the data analysis show that, the AULTE has the lowest MSE values among the examined estimators. In Figure \ref{plot:app}, we also plot the MSE graphs of the estimators for a fixed value of the parameter $k$ which is estimated by $k=1/(\ahat^{\top}\ahat)$ for each method where $\ahat = \bQ^{\top}\mle$ and for a range of the parameter values of $d$.
Furthermore, we observe that MAULTE has the lower MSE value than LTE and MLE for specific values of $d$. Thus, we recommend that AULTE and MAULTE are good alternatives to MLE and LTE in the Bell regression model when there is multicollinearity.
\section{Conclusion}
This paper aims to derive two new estimators, called AULTE and MALTE in the literature, as an alternative to LTE and MLE in the Bell regression model. One of the advantages of the proposed estimators is their superiority over well-known estimators such as LTE and MLE in terms of MSE. We prove this superiority both theoretically and numerically through theoretical comparisons, simulations and real-life examples. As a result of the simulation study, we recommend MAULTE in the Bell regression model due to its performance with respect to the MSE criteria. Moreover, the practical data analysis results show that the MSE of AULTE is lower than its competitors MLE, LTE and MAULTE. We conclude that the estimators proposed in this paper (AULTE and MAULTE) are valuable alternatives to LTE and MLE in the Bell regression model.

\label{sec:conc}

\bigskip
\noindent\textbf{Acknowledgements}
This study was supported by TUBITAK 2218-National Postdoctoral Research Fellowship Programme with project number 122C104.
\newline
\bigskip
\noindent\textbf{Author Contributions} Caner Tanış: Intoduction, Methodology, Simulation, Real data application, Writing-original draft. Yasin Asar: Methodology, Simulation, Real data application, Writing-reviewing \& editing
\newline
\bigskip
\noindent\textbf{Funding} The authors declare that they have no financial interests.
\newline
\bigskip
\noindent\textbf{Data Availability} The dataset supports the findings of this study are openly available in reference list.
\newline
\bigskip
\noindent\textbf{Declarations}
\newline
\bigskip
\noindent\textbf{Conflict of interest} All authors declare that they have no conflict of interest.
\newline
\bigskip
\noindent\textbf{Ethics statements} The paper is not under consideration for publication in any other venue or language at this time.



\end{document}